\documentclass{amsart}              
\usepackage{amssymb,amsmath}
\usepackage[T1]{fontenc}
\usepackage[utf8]{inputenc}
\usepackage{enumitem}
\usepackage{hyperref}
\usepackage[labelfont=rm]{caption}
\usepackage[defaultlines=3,all]{nowidow}
\usepackage{tikz}

\usepackage{pgf}
\usepackage{xcolor}
\usepackage{nicematrix}

\usetikzlibrary{braids,
shapes, matrix, arrows, knots, arrows.meta, 3d, pgfplots.groupplots, angles, ext.transformations.mirror,
calc, decorations.markings, intersections, cd}

\DeclareMathOperator{\sing}{Sing}

\DeclareMathOperator{\disc}{disc}
\DeclareMathOperator{\ctr}{Center}
\DeclareMathOperator{\aut}{Aut}
\newcommand{\bc}{\mathbb{C}}
\newcommand{\bp}{\mathbb{P}}
\newcommand{\br}{\mathbb{R}}
\newcommand{\aff}{\text{\rm aff}}
\newcommand{\proy}{\text{\rm proj}}
\newcommand{\abs}[1]{\left\lvert #1\right\rvert}
\newcommand{\ns}[1]{\left\langle\!\left\langle #1\right\rangle\!\right\rangle}
\newtheorem{thm}{Theorem}[section]
\newtheorem{prop}[thm]{Proposition}
\newtheorem{propdef}[thm]{Proposition-Definition}
\newtheorem{cor}[thm]{Corollary}
\newtheorem{lema}[thm]{Lemma}

\theoremstyle{definition}
\newtheorem{dfn}[thm]{Definition}

\newtheorem{ejm}[thm]{Example}

\theoremstyle{remark}
\newtheorem{remark}[thm]{Remark}

\numberwithin{equation}{section}

\begin{document}
\title[Topology of curves]{Topology of complex plane curves: braid monodromy, local and global problems}    \author[E. Artal]{Enrique Artal Bartolo}
\address[E.~Artal]{Departamento de Matem\'{a}ticas, IUMA \\
Universidad de Zaragoza \\
C.~Pedro Cerbuna 12, 50009, Zaragoza, Spain}
\urladdr{\url{http://riemann.unizar.es/~artal}}
\email{\href{mailto:artal@unizar.es}{artal@unizar.es}}

\thanks{\noindent Partially supported by PID2024-156181NB-C33 funded by 
MICIU/AEI/10.13039/ 501100011033 and by FEDER, UE}
\date{\today}

\dedicatory{A mi amigo Alejandro con el que he construido m{\'a}s que matem{\'a}ticas.}



\date{}

\begin{abstract}
The embeddings of complex plane projective curves in the plane 
are a cornerstone of the topological study of algebraic varieties.
In this work, we deal with the local and global aspects of these 
embeddings, with a special attention to its historical progress.
\end{abstract}

\maketitle
\thispagestyle{empty}

\section*{Introduction}

Topology of algebraic varieties is an active field of study since the nineteenth century. Classical Riemann works showed that in the one-dimensional case, 
algebraic varieties can be seen as complex \emph{multivalued} functions with some ramification points, leading to a starting point
of an induction process. The development of the theory of fundamental groups and coverings showed that this is equivalent to the study 
of finite representations of free groups, since they are the fundamental groups of the punctured Riemann sphere or the punctured complex line.

At the beginning of the twentieth century, Lefschetz~\cite{Lef:50} added some strong bricks for this induction process: homological invariants
up to some dimension could be studied using \emph{generic} hyperplane sections. While the homotopic counterpart of these results must wait for several
decades, the particular case of fundamental group is settled quite soon by \cite{Zariski1937} (see~\cite{HL:73} for a modern proof):
\emph{The fundamental group of a (quasi-)projective smooth variety
of dimension $n$ coincides
with the one of its hyperplane section if $n>2$; the injection morphism is surjective on fundamental groups if $n=2$. }

The combination
of this result with the projection techniques yields Zariski's foundational paper~\cite{Zariski1929}.
In the  twenties and thirties
there is an active development on this topic. Van Kampen~\cite{Kampen1933} formalized Zariski's great ideas and 
set the groundwork 
for the celebrated Seifert-van Kampen Theorem for the computation of fundamental groups. It is not an isolated case that the
techniques used in this field became standard in geometric topology. Zariski has a series of papers 
that made a significant
progress in the study
of the topology of plane curve complements, which became a subject on its own.

As a summary we present some of Zariski's results. In~\cite{Zariski1931}, he relates the first homology groups of cyclic covers 
of the plane ramified along a curve with the \emph{position} of its singularities; these key ideas were developed much later by other
authors~\cite{Esnault1982, Esnault1982a, Libgober1982a, Esnault1987, Loeser1990, Artal1994}, see~\cite{li:11} for a more complete survey on Alexander invariants.
This paper has its origins in the so-called Zariski sextics of~\cite{Zariski1929}, where the fact that their six cusps
were on a conic seemed to have a topological consequence, as it was confirmed in~\cite{Zariski1931}.
There is a less known paper~\cite{zr:29a} which also deals with cyclic covers. If one looks carefully, what Zariski
states in loc.~cit. in modern terms looks like this: \emph{the Alexander polynomial of an irreducible projective curve 
whose degree is a prime power is~$1$}. Actually, Zariski's proof implies a stronger
result: \emph{the Alexander polynomial of an irreducible curve 
cannot have roots of unity whose order is a prime power}.

A few years later, in~\cite{Zariski1936}, he computed the fundamental group of the complement of maximal cuspidal rational
curves, i.e., the dual of nodal rational curves. An interesting point is their relationship with the braid groups of the $2$-sphere.
Actually, in~\cite{Zariski1937a}, the first presentation of braid groups over a closed orientable surface of genus~$g$ is given.
He computes the fundamental group of the dual of a smooth cubic (a sextic with nine cusps), and using nice deformation
arguments he finds the first example of what will be called \emph{Zariski pairs}, ending a guess started in \cite{Zariski1929}
(where he proved that the fundamental group of the complement of a sextic with six cusps not on a conic is not isomorphic
to this group when the cusps are on a conic).

There is also some activity in the local case, relating algebraic curves with knot theory. An \emph{algebraic knot} is the intersection of an algebraic curve with a small enough ball centered at a point of the curve. Its topological type does not depend on the radius (if small enough)
and it detects if a point is singular or smooth. Wirtinger~\cite{wirt:05} gave a method to compute the fundamental group of an algebraic knot
which has become standard for arbitrary knots; the relationship with iterated torus knots was studied by Brauner~\cite{Brauner1928}.
A nice exposition for algebraic knots and links is given in~\cite{Brieskorn1986}; another nice topological exposition~\cite{Michel-Weber}
by Michel and Weber is unfortunately unpublished.

In the thirties, O.~Chisini~\cite{Chisini1933} realized that the ideas of Zariski and van Kampen contained a stronger invariant
than the fundamental group, which became the so-called \emph{braid monodromy} by B.~Moishezon, who applied it to the topological
study of complex surfaces, see~\cite{Moi:81} and later works. The topological aspects of braid monodromy were also developed
by A.~Libgober~\cite{Libgober1986}, see Theorem~\ref{thm:lib}, from a homotopic point of view. Later on, Kulikov and Teicher~\cite{kt:00}
(for curves with double points), and Carmona~\cite{Carmona2003} (in the general case), showed that braid monodromy determined
the topology of the embedding of a curve in the plane. In a joint work with J.~Carmona, J.I. Cogolludo~\cite{acc:03, acc:07} (and in other work~\cite{accm:05} with M.~Marco), we proved a kind of converse which was used to prove 
the existence of new Zariski pairs (including some of them which cannot be distinguished algebraically).

One of the main problems of braid monodromy is that it is extremely hard to compute. Salvetti~\cite{Salvetti1988} gave a general method
for its construction for the complexification of real line arrangements. A related invariant, the \emph{wiring diagram}, was developed by Arvola
to be used for general line arrangements, but it is more difficult to be used, see e.g~\cite{accm:07, gb:16}. 
Starting from~\cite{Carmona2003}, a computer program was built for its computation, see also~\cite{Bessis2005} for another approach. 
From these ideas, Marco and Rodr{\'i}guez~\cite{mr:16} constructed a new method implemented in \texttt{Sagemath}~\cite{Sagemath} which guarantees the correctness of the output, when produced, though it uses numerical methods.

This paper is organized as follows. In~\S\ref{sec:braid}, we introduce the main definitions and results about braid groups.
There is a geometric action of braid groups on free groups seen as fundamental groups of punctured real planes which is crucial 
for braid monodromy; the main properties of this action and its notation is settled in~\S\ref{sec:free}.
In \S\ref{sec:local} we show the local structure of braid monodromy related with Weierstra{\ss} and Puiseux theory.
In the short section~\S\ref{sec:semilocal}, it is explained how to \emph{glue} the local braids of several singular points
of the projection in the same vertical line. In \S\ref{sec:global} we define the braid monodromy for the projection of an affine curve
with no vertical asymptotes, without any further genericity conditions; Hurwitz moves are defined.
Generic braid monodromy factorizations for projective curves are defined in \S\ref{sec:proj}; the relationship between the fundamental
group of the complement of a projective curve and a generic affinization is given.
In the last section \S\ref{sec:app}, some applications are given, mainly braid monodromy factorizations associated to Kummer covers and some words about braid monodromy factorizations of symplectic curves with special emphasis on some results by Moishezon.
 \section{Braids}\label{sec:braid}

The foundation of braid theory is laid in \cite{Artin1925, Artin1947} by Artin, 
where the 
classic presentation of the braid group~$\mathbb{B}_n$ is given:
\begin{equation}
\label{eq:artin}
\left\langle
\sigma_1,\dots,\sigma_{n-1}
\middle|\,
\underset{j-i>1}{[\sigma_i,\sigma_j]=1},\
\underset{1\leq i<n-1}{\sigma_i\cdot\sigma_{i+1}\cdot\sigma_i=\sigma_{i+1}\cdot\sigma_i\cdot\sigma_{i+1}}
\right\rangle.
\end{equation}
Braids can be seen as homotopy classes of sets of continuous maps
$\{\gamma_i:[0,1]\to\mathbb{C}_{i=1}^n\}$ such that $\{\gamma_i(t_0)\mid i=1,\dots,n\}=\{1,\dots,n\}$
for $t_0\in\{0,1\}$, and $\gamma_i(t)\neq\gamma_j(t)$ if $i\neq j$, $\forall t\in [0,1]$.
They are represented by their graphs in $\mathbb{C}\times[0,1]$. 
It is customary
to represent them in $\br\times[0,1]$ using the projection $\mathbb{C}\to\mathbb{R}$ by taking the real part.
The overpass of a crossing corresponds to the \emph{strand} with smaller imaginary part.
The standard Artin generators are shown in Figure~\ref{fig:gen}. 
The product of braids are done by putting one above the other and rescaling. It is clear that 
after a homotopy one can get only double points, which shows that $\sigma_1,\dots,\sigma_{n-1}$
are indeed a generator system.

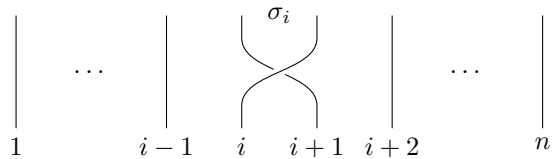
\begin{figure}[ht]
\begin{tikzpicture}
\pic[yscale=-1,
braid/.cd,
number of strands=4] {braid={s_2}};
\node[below] at (-2, 0) {$1$};
\node[below] at (0, 0) {$i-1$};
\node[below] at (1, 0) {$i$};
\node[below] at (2, 0) {$i+1$};
\node[below] at (3, 0) {$i+2$};
\node[below] at (5, 0) {$n$};
\draw (-2,0) -- (-2, 1.5);
\draw (5,0) -- (5, 1.5);
\node at (4, .75) {$\dots$};
\node at (-1, .75) {$\dots$};
\node at (1.5, 1.5) {$\sigma_i$};
\end{tikzpicture}
 \caption{Standard Artin generator}
\label{fig:gen}
\end{figure}

It is easy to see that the generators satisfy the relations in \eqref{eq:artin}, see
Figure~\ref{fig:rel} to see how they look. 
There is a natural morphism $\mathbb{B}_n\to\mathfrak{S}_n$ (the $n$-symmetric group) given 
by the permutation of the end points; in terms of the generators, the map is given by
$\sigma_i\mapsto(i, i+1)$. Its kernel $\mathbb{P}_n$ is the \emph{pure braid group}
and it consists of the braids such that $\gamma_i(0)=\gamma_i(1)$, which can be assumed
to be~$i$ after a reordering.

\begin{figure}[ht]
\begin{tikzpicture}[scale=.8]
\pic[yscale=-1,
braid/.cd,
number of strands=4] {braid={s_1 s_3}};
\node at (1.9,-.5) {$\sigma_1\cdot\sigma_3$};

\node at (4.2,1.25) {$=$};

\begin{scope}[xshift=4.5cm]
\pic[yscale=-1,
braid/.cd,
number of strands=4] {braid={s_3 s_1}};
\node at (1.9,-.5) {$\sigma_3\cdot\sigma_1$};
\end{scope}

\begin{scope}[xshift=9cm]
\pic[yscale=-1,
braid/.cd,
number of strands=3] {braid={s_1 s_2 s_1}};
\node at (1.1,-.5) {$\sigma_1\cdot\sigma_2\cdot\sigma_1$};
\node at (2.75,1.75) {$=$};
\end{scope}

\begin{scope}[xshift=12cm]
\pic[yscale=-1,
braid/.cd,
number of strands=3] {braid={s_2 s_1 s_2}};
\node at (1.1,-.5) {$\sigma_2\cdot\sigma_1\cdot\sigma_2$};
\end{scope}

\end{tikzpicture}
 \caption{Artin relations}
\label{fig:rel}
\end{figure}
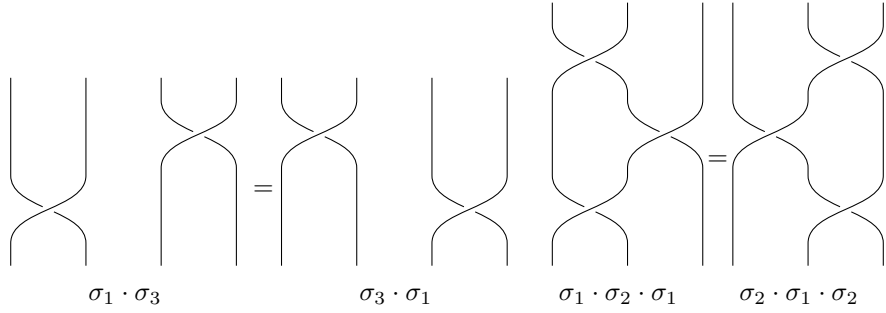

It is useful to interpret the braid and pure braid groups as fundamental groups.
Let
$X_n:=\{(x_1,\dots,x_n)\in\bc^n\mid \#\{x_1,\dots,x_n\}=n\}$,
the complement of the braid arrangement. The symmetric group
$\mathfrak{S}_n$ acts naturally on $X_n$ and we are interested in
its quotient $Y_n:=X_n/\mathfrak{S}_n$.
Let us interpret $Y_n$ in an another way. Let
\[
V_n:=\{p(t)\in\bc[t]\mid\deg p=n, p \text{ monic}\}.
\]
This space can be identified with $\bc^n$, via the coefficients. The map
\[
\begin{tikzcd}[row sep=0,/tikz/column 1/.append style={anchor=base east},/tikz/column 2/.append style={anchor=base west}]
\bc^n\rar["\Phi"] & V_n\\
(x_1,\dots,x_n)\rar[mapsto]&\displaystyle\prod_{j=1}^n (t -x_i)
\end{tikzcd}
\]
realizes the quotient map by the action of the symmetric group in $\bc^n$. Let 
\[
T_n:=\{p(t)\in V_n\mid p \text{ has only simple roots}\}.
\]
Then, the restriction $\Phi_{|}:X_n\to T_n$ is a model of the quotient
$X_n\to Y_n$. Note also that $T_n=V_n\setminus D_n$ where $D_n$
is the zero locus of the general discriminant polynomial. 
From now on we will not distinguish between $Y_n$ and $T_n$.
The projective
counterpart is the content of~\cite{Zariski1936}.

In a natural way, given $\mathbf{x}^0:=(x_1^0,\dots,x_n^0)\in X_n$ and
$\{x_1^0,\dots,x_n^0\}\equiv p^0(t):=\Phi(\mathbf{x}^0)\in T_n$, we have that the braid group
and the pure braid group are fundamental groups:
\[
\mathbb{B}_n\cong\pi_1(T_n; p^0),\qquad \mathbb{P}_n\cong\pi_1(X_n; \mathbf{x}^0).
\]
This identification is achieved fixing $\mathbf{x}^0:=(1,\dots,n)$.
There is a way to represent braids in the same way when $\mathbf{x}^0\neq(1,\dots,n)$.
Let us start first with the case where the real parts of $x_1,\dots,x_n$ are pairwise distinct.
In that case the same kind of projections allows to describe 
a braid as word in the standard Artin generators.
For arbitrary $\mathbf{x}^0$, when we have points with the same real part
we slightly deform the projection as in Figure~\ref{fig:proj} such that the bigger imaginary part
is moved to the right. Note that other conventions may work but it is important
to fix one.

\begin{figure}[ht]
\begin{tikzpicture}

\draw[] (0,2) -- (0,-1) -- (10,-1);
\fill (5,0) node[left] {$0$} circle [radius=.1];
\fill (7,0) node[right] {$1$} circle [radius=.1];
\fill (5,2) node[above] {$\sqrt{-1}$} circle [radius=.1];

\draw[->] (-1,.5) node[above] {$\mathbb{C}$} -- (-1,-2)
node[below] {$\mathbb{R}$} ;
\draw (0, -2.2) -- (10,-2.2);
\foreach \x in {.2, .6, ..., 3.6,7.6, 8, ..., 9.6}
{
\draw (\x,-1) -- (\x,2);
\draw[dotted] (\x,-1) -- (\x,-2.2);
}
\foreach \x in {0, .4, ...,1.2}
{
\draw (3.8 + 1.2*\x,-1) -- (3.8 + \x,2);
\draw[dotted] (3.8 + 1.2*\x,-1) -- (3.8 + 1.2 * \x,-2.2);
}
\foreach \x in {2, 1.6, ..., 0 }
{
\draw (7.6-\x,-1) -- (7.6 - 1.1 * \x,2);
\draw[dotted] (7.6-\x,-1) -- (7.6 - \x,-2.2);
}
\end{tikzpicture}
 \caption{Conventions about projections}
\label{fig:proj}
\end{figure}
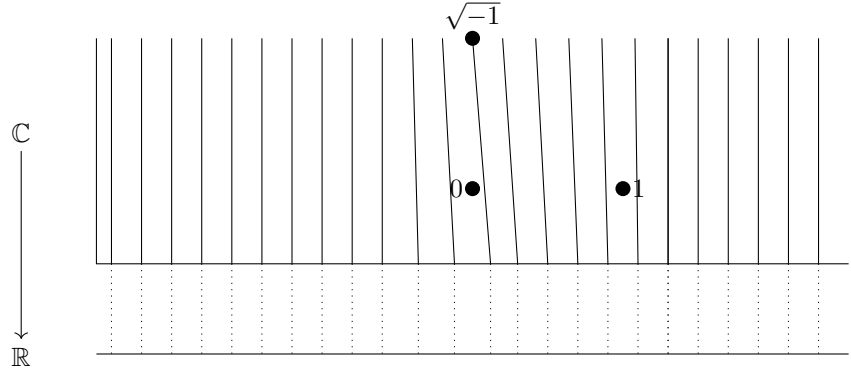

The same conventions can be applied when we consider an element
of the fundamental grupoid of $Y_n$, i.e., the homotopy
class of a path in $Y_n$ where the end points are not equal.
We can then identify any braid (with equal or distinct end points)
with an element in $\mathbb{B}_n$ as $\pi_1(Y_n;\{1,\dots,n\})$. 

\begin{remark}\label{rem:lex}
The above identification can also be understood from the choice of 
fixed homotopy classes of paths between elements of $Y_n$.
We index lexicographically $\{x_1,\dots,x_n\}$:
\[
\Re x_j< \Re x_{j+1}\text{ or } \Re x_j = \Re x_{j+1}\text{ and }\Im x_j < \Im x_{j+1}.
\]
Once such a braid is generically drawn using this projection, the order of the crossings
allows to write a word in the Artin generators.
\end{remark}

There is another way to identify $\pi_1(Y_n;\{x_1,\dots,x_n\})$ with 
$\mathbb{B}_n$, adding some extra data. We start an order of $\{x_1,\dots,x_n\}$ and 
we choose a 
piecewise $\mathcal{C}^{(1)}$ simple path $\Gamma$ from $x_1$ to $x_n$ passing through
all the points in the given order which is $\mathcal{C}^{(1)}$ at $x_i$. We denote by $\Gamma_i$, $i=1,\dots,n-1$, the subpaths
from $x_i$ to $x_{i+1}$.

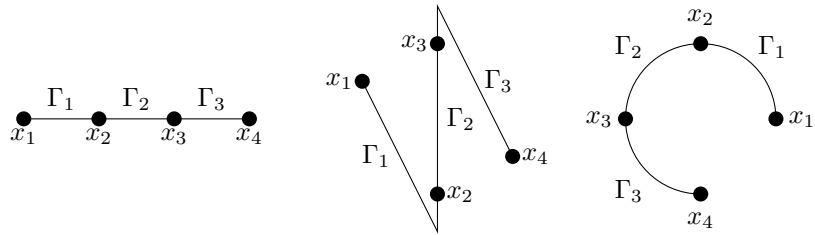
\begin{figure}[ht]
\centering
\begin{tikzpicture}
\begin{scope}
\draw (1,0) -- (4,0);
\foreach \x in {1,...,4}
{
\fill (\x,0) node[below] {$x_\x$} circle [radius=.1cm];
}
\foreach \x in {1,...,3}
{
\node[above] at ({\x+.5},0) {$\Gamma_\x$};
}
\end{scope}

\begin{scope}[xshift=4.5cm]
\coordinate (A1) at (1, 1/2);
\coordinate (A2) at (2, -1);
\coordinate (A3) at (2, 1);
\coordinate (A4) at (3, -1/2);
\coordinate (B1) at (1, 1/2);
\coordinate (B2) at (2, -1.5);
\coordinate (B3) at (2, 1.5);
\coordinate (B4) at (3, -1/2);
\fill (A1) node[left] {$x_1$} circle [radius=.1cm];
\fill (A2) node[right] {$x_2$} circle [radius=.1cm];
\fill (A3) node[left] {$x_3$} circle [radius=.1cm];
\fill (A4) node[right] {$x_4$} circle [radius=.1cm];
\draw (B1) -- node[left, pos=.5] {$\Gamma_1$} (B2)
-- node[right] {$\Gamma_2$}(B3)
-- node[right] {$\Gamma_3$}(B4);
\end{scope}

\begin{scope}[xshift=10cm]
\draw (1,0) arc [start angle=0,end angle=270,radius=1cm];
\foreach \x in {1,...,4}
{
\coordinate (X\x) at ({90*(\x -1)}:1);
\coordinate (Y\x) at ({90*(\x -1)+45}:1.35);
\fill (X\x) circle [radius=.1cm];
\node at ($1.35*(X\x)$) {$x_\x$};
}
\foreach \x in {1,...,3}
{
\node at (Y\x) {$\Gamma_\x$};
}
\draw ;
\end{scope}

\end{tikzpicture}
 \caption{Examples of paths $\Gamma$.}
\label{fig:path}
\end{figure}

The identification is done using the following braids for $\sigma_i$
as paths $\gamma:=(\gamma_1,\dots,\gamma_n)\!:\![0,1]\to X_n$ defining loops in $Y_n$, where $\gamma_j$ is constant for $j\neq i, i+1$, $\gamma_i$ runs from $x_i$ to $x_{i+1}$ to the \emph{right}
of $\Gamma_i$ and $\gamma_{i+1}$ runs from $x_{i+1}$ to $x_i$ to the \emph{left}
of $\Gamma_i$. The usual identification of $\mathbb{B}_n$ with $\pi_1(Y_n;\{1,\dots,n\})$
corresponds to the left-hand side of Figure~\ref{fig:path}; the middle part corresponds
with the identification with $\pi_1(Y_n;\{x_1,\dots,x_n\})$ using the lexicographic order 
as in Remark~\ref{rem:lex}. The picture in the right-hand side shows another path
which will be useful later.

 \section{Free Groups}\label{sec:free}

One main feature of free groups (with a fixed basis) is their connection with the topology of a punctured plane.
Fix $\mathbf{x}:=(x_1,\dots,x_n)\in X_n$ (the order will somehow be important); in order to not overload
the notation the set $\{x_1,\dots,x_n\}$ will also be denoted as $\mathbf{x}$ for short.
For $x_0\in\bc\setminus\mathbf{x}$,  
$\mathbb{F}_n:=\pi_1(\bc\setminus\mathbf{x};x_0)$ is a free group and we are going to describe a natural family of bases.

Before recalling some classical definitions, we introduce a notation. 
By $\mathbb{D}_r(x_i)$ (resp. $\mathbb{S}_r(x_i)$) we mean the disk (resp. the circle)
of radius~$r$ centered at~$x_i\in\mathbb{C}$; if the center is~$0$ they will be
denoted as $\mathbb{D}_r$ and $\mathbb{S}_r$.

\begin{dfn}\label{def:meridian}
An element $\mu\in\mathbb{F}_n$ is a \emph{meridian} of $x_i$ if it admits a representative as follows.
Let $\alpha$ be a path joining $x_0$ and $x'_i\in\partial\mathbb{D}_r(x_i)$ in 
$\mathbb{C}\setminus\mathbf{x}$, such that $x_j\notin\mathbb{D}_r(x_i)$
if $j\neq i$. Let $\delta$ be the loop based at $x'_i$ that runs counterclockwise along $\partial\mathbb{D}_r(x_i)$.
Then $\mu$ is the class of $\alpha\cdot\delta\cdot\alpha^{-1}$ (where $\alpha^{-1}$ refers to the opposite path to $\alpha$). The pair $(\alpha,\delta)$
is called a \emph{decomposition of the meridian}~$\mu$. Note that the meridians of $x_i$
form a conjugacy class of $\mathbb{F}_n$.
\end{dfn}

We can see $\bc$ as a subset of $\bp^1\equiv\bc\cup\{\infty\}$. In this way we can define also
the \emph{meridians} of $\infty$. The subpath $\delta$ of a meridian of $\infty$ runs clockwise since the 
positive orientation is defined as the boundary of a \emph{disk} centered at $\infty$. 

\begin{dfn}
An ordered basis $\mu_1,\dots,\mu_n$ of $\mathbb{F}_n$ is called \emph{pseudogeometric}
if there exists a permutation $\sigma\in\mathfrak{S}_n$ such that $\mu_i$ is a meridian
of $x_{i^\sigma}$ and $(\mu_1\cdot\ldots\cdot\mu_n)^{-1}$ is a meridian of~$\infty$.
\end{dfn}

In the case that $x_0$ is in the boundary of a disk $\mathbb{D}$ such that $\mathbf{x}\subset\mathring{\mathbb{D}}$,
there is a special meridian of $\infty$, namely the class of the loop based at $x_0$ that runs 
clockwise along $\partial\mathbb{D}$.

\begin{dfn}
If $x_0$ is in the boundary of a disk $\mathbb{D}$ such that $\mathbf{x}\subset\mathring{\mathbb{D}}$,
an ordered basis $\mu_1,\dots,\mu_n$ of $\mathbb{F}_n$ is called \emph{geometric}
if there exists a permutation $\sigma\in\mathfrak{S}_n$ such that $\mu_i$ is a meridian
of $x_{i^\sigma}$ and $\mu_1\cdot\ldots\cdot\mu_n$ is the class of the loop based at $x_0$ that runs 
counterclockwise along $\partial\mathbb{D}$.
\end{dfn}

\begin{ejm}
Let us fix such a disk $\mathbb{D}$ and $x_0\in\partial\mathbb{D}$; consider also $r>0$ such that 
the disks $\mathbb{D}_r(x_i)$ are pairwise disjoint and contained in $\mathring{\mathbb{D}}$.
Let $\Gamma$ be a path as in Figure~\ref{fig:path} (or Figure~\ref{fig:base}) 
for $\mathbf{x}$ and let $\tilde{\Gamma}_0$ be a path 
in $\mathbb{D}\setminus\Gamma$ joining $x_0$ and $x'_i\in\partial\mathbb{D}_r(x_1)$.
Inductively, for $i\in\{1,\dots,n-1\}$ we consider paths~$\tilde{\Gamma}_i$ (to the \emph{left} of $\Gamma_i$) from 
$x'_i$ to $x'_{i+1}\in\partial\mathbb{D}_r(x_{i+1})$. Let us denote $\delta_i$ the 
loop based at $x'_i$ that runs 
counterclockwise along $\partial\mathbb{D}_r(x_i)$. Then, it is not difficult to check that
\[
\mu_i:=\left(\prod_{j=0}^{i-1}\tilde{\Gamma}_j\right)\cdot\delta_j\cdot\left(\prod_{j=0}^{i-1}\tilde{\Gamma}_j\right)^{-1},\quad 
i=1,\dots,n,
\]
is a geometric basis of $\mathbb{F}_n$.
\end{ejm}

\begin{figure}[ht]
\centering
\begin{tikzpicture}
\coordinate (A0) at (-2,0);
\coordinate (A1) at (-1,1);
\coordinate (A2) at (1,-1);
\coordinate (A3) at (1,1);
\coordinate (A4) at (3,0);
\coordinate (A-1) at ($(A0)-(2,0)$);

\foreach \x in {-1,...,4}
{
\fill (A\x) circle [radius=.1cm];
}
\foreach \x in {0,...,4}
{
\draw (A\x) circle [radius=.25cm];
}
\draw(A-1);
\draw (A-1) node[left] {$x_0$} to[out=0, in=180] node[above] {$\tilde\Gamma_0$} ($(A0)-(.25,0)$);
\draw (A0) node[below=6pt] {$x_1$}
to[out=0, in=180] node[right, pos=.25] {$\Gamma_1$}
(A1) node[above=6pt] {$x_2$}
to[out=0, in=180] node[pos=.5, left] {$\Gamma_2$} (A2) node[below=6pt] {$x_3$} to[out=0, in=-90]node[pos =.5, right] {$\Gamma_3$} (A3) node[below right=6pt] {$x_4$} to[out=90, in=90]node[below] {$\Gamma_4$} (A4) node[right=6pt] {$x_5$};

\draw ($(A0)-(.25,0)$) to[out=90, in=180] node[above left,pos=.5] {$\tilde\Gamma_1$}  ($(A1)+(0,.25)$)
to[out=0, in=180] node[right,pos=.25] {$\tilde\Gamma_2$}  ($(A2)+(0,.25)$)
to[out=0, in=180] node[left,pos=.5] {$\tilde\Gamma_3$}  ($(A3)+(0,.25)$)
to[out=90, in=90] node[right=5pt,pos=.5] {$\tilde\Gamma_4$}  ($(A4)+(0,.25)$);
\end{tikzpicture}
 \caption{\emph{Lexicographic} geometric basis}
\label{fig:base}
\end{figure}
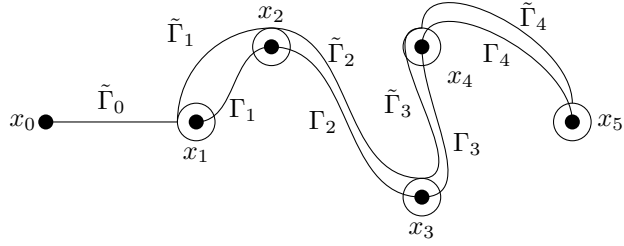

There is a natural action of $\mathbb{B}_n$ on $\mathbb{F}_n$. It can be understood from the fact that
the braid group is also the group of isotopy classes of homeomorphisms of $\bc$ globally fixing 
$\{x_1,\dots,x_n\}$
and being the identity outside of a compact set. There is another interpretation. 
Given a braid $\tau\in\pi_1(Y_n;\mathbf{x})$, let $X_\tau:=(\bc\times[0,1])\setminus\tau$ (where $\tau$ is seen as 
the graph of the $n$-maps which define it). The maps $i_j:\bc\setminus\mathbf{x}\to X_\tau$ defined 
as $i_j(z):=(z,j)$, for $j=0,1$, are homotopy equivalences. Hence, we can consider the composition $\Phi_\tau$
\[
\begin{tikzcd}
\pi_1(\bc\setminus\mathbf{x};x_0)\rar["i_{0,*}"]\ar[rrr, bend right=10pt, "\Phi_\tau" below]&\pi_1(X_\tau;(x_0,0))\rar["\alpha"]&\pi_1(X_\tau;(x_0,1))\rar["i_{1,*}^{-1}"]&\pi_1(\bc\setminus\mathbf{x};x_0)
\end{tikzcd}
\]
where $\alpha$ is the  change of base point isomorphism defined by the path $t\mapsto(x_0,t)$;
it is well defined if the representative of $\tau$ avoids this path, which can always be obtained
either by a homotopy or by a new choice of~$x_0$. 
If $\gamma\in\pi_1(\bc\setminus\mathbf{x};x_0)$ and $\tau\in\pi_1(Y_n;\mathbf{x})$, then 
$\gamma^\tau:=\Phi_\tau(\gamma)$.

If $\pi_1(\bc\setminus\mathbf{x};x_0)$ is identified with $\mathbb{F}_n$ and $\pi_1(Y_n;\mathbf{x})$ 
is identified with $\mathbb{B}_n$ via the same path~$\Gamma$ then this action is given by:
\[
\mu_i^{\sigma_j}=
\begin{cases}
\mu_i\cdot\mu_{i+1}\cdot\mu_i^{-1}&\text{if }i=j\\
\mu_{i-1}&\text{if }i=j+1\\
\mu_i&\text{otherwise.}
\end{cases}
\]
We can put this together as an action
\begin{equation}\label{eq:action}
\begin{tikzcd}[row sep=0pt]
\mathbb{F}_n\times\mathbb{B}_n\rar[]&\mathbb{F}_n\\
(\mu,\tau)\rar[mapsto]&\mu^\tau.
\end{tikzcd}
\end{equation}
This action has two main properties; let $\phi:=\phi_\tau$ be the automorphism defined by the action of $\tau$
\begin{enumerate}[label=\rm(A\arabic{enumi})]
\item\label{a1} If $\rho:\mathbb{B}_n\to\mathfrak{S}_n$ is the natural epimorphism defined by the ends of a braid,
let $\sigma:=\rho(\tau)$. Then
$\phi(\mu_i)$ is a meridian of $x_{i^{\sigma}}$, i.e., a conjugate of $\mu_{i^\sigma}$.
\item\label{a2} $\phi(\mu_1\cdot\ldots\cdot\mu_n)=\mu_1\cdot\ldots\cdot\mu_n$.
\end{enumerate}

The main result in \cite{Artin1947} says that this action defines an anti-monomorphism $\mathbb{B}_n\to\aut\mathbb{F}_n$, whose 
image is formed by the automorphisms satisfying \ref{a1} and \ref{a2} and it can also be reinterpreted as follows. We consider a base point
$x_0$ for which 
the definition of the geometric basis stands.

\begin{thm}\label{thm:artin_action}
Let $\mathcal{B}_n:=\{(\mu_1,\dots,\mu_n)\mid\text{geometric basis of }\mathbb{F}_n\}$. The action of 
$\mathbb{B}_n$ on $\mathbb{F}_n$ induces a free and transitive action on $\mathcal{B}_n$.
\end{thm}

 \section{Topology of the germ of a plane curve singularity}
\label{sec:local}

A germ $(C,0)\subset(\bc^2,0)$ is defined as the zero locus 
of a square-free non-constant germ $f(x,y)\in\bc\{x,y\}$,
the ring of convergent power series. 
The analytic study of these germs is based on Weierstra\ss{} theory.
There are many texts where is developed, see e.g. \cite{Griffiths1978, GuRo};
in the case we are interested in, a nice exposition is done in \cite{Gr:89}.

\begin{dfn}
A polynomial $p(x,y)\in\bc\{x\}[y]$ is said to be a \emph{Weierstra\ss{} polynomial}
if it is monic and the other coefficients vanish at $x=0$.
A series $f(x,y)\in\bc\{x,y\}$ is \emph{$y$-regular
of order $n>0$} if  $f(0,y)=y^n u(y)$, where $u(0)\neq0$, i.e., $u$
is a unit in $\bc\{y\}\subset\bc\{x,y\}$.
\end{dfn}

\begin{thm}[Weierstra\ss{} Preparation Theorem, {\cite[p.~8]{Griffiths1978}}]
Let $f(x,y)\in\bc\{x,y\}$ be $y$-regular
of order $n>0$. Then, there exist a
unique Weierstra\ss{} polynomial $p(x,y)\in\bc\{x\}[y]$ of degree~$n$ 
and a unit $u(x,y)$ such that
$f=pu$. In particular the germs of zero loci of $p$ and $f$ coincide.
\end{thm}

This theorem and related ones, such as the Weierstra\ss{} Division Theorem,
have important consequences.

\begin{cor}
\mbox{}
\begin{enumerate}[label=\rm(\roman{enumi})]
\item The ring $\bc\{x,y\}$ is factorial and Noetherian.
\item A Weierstra\ss{} polynomial is irreducible in $\bc\{x,y\}$
if and only if it is irreducible in $\bc\{x\}[y]$.
\item Let $C$ be the germ of the zero locus of $f\in\bc\{x,y\}$.
Then $f$ has only one irreducible factor if and only if $C\setminus\{0\}$
is connected.
\item Let $p\in\bc\{x\}[y]$ be a square-free Weierstra\ss{} polynomial.
Then, there exists $d\geq 0$ such that $\disc_y p\in\bc\{x\}$ is
of the form $x^d v(x)$, $v(0)\neq 0$.
\end{enumerate}
\end{cor}

These results provide a way to get \emph{good representatives} of a germ. We have
chosen a way which is very dependent on the coordinates (using polydisks). The topological
properties of the result will imply that an analytic change of coordinates will lead to similar results.

\begin{propdef}\label{propdef:ed}
Let $f\in\bc\{x,y\}$ be a reduced element and let $C$ be the germ of its zero locus.
A \emph{good representative} of $f$ is a holomorphic function 
defined in a neighbourhood of $\mathbb{D}^2_\delta\times\mathbb{D}^2_\varepsilon$,
$0<\delta\ll\varepsilon$,
(which by abuse of notation will still be denoted by $f$
if no ambiguity is likely to arise), satisfying:
\begin{enumerate}[label=\rm(\alph{enumi})]
\item $f(0,y)\neq 0$ if $0<\abs{y}\leq\varepsilon$;

\item $f(x,y)\neq 0$ if $\abs{y}=\varepsilon$ and $\abs{x}\leq \delta$;

\item these properties also hold for the Weierstra\ss{} polynomial $p$ associated to $f$
(which defines the same germ~$p$).

\item $\disc_y p$ does not vanish if $0<\abs{x}\leq \delta$; in particular the point $(0,0)$
is the zero locus of $\{p(x,y),\frac{\partial p}{\partial y}(x,y)\}$.
\end{enumerate}

Any such $f$ admits good representatives. The topological type of the pair 
$(\mathbb{D}^2_\delta\times\mathbb{D}^2_\varepsilon, C)$ does not depend
on the particular good representative.
\end{propdef}

Let us relate this discourse with braids. Let us fix a complex number $\hat{\delta}\in\mathbb{S}^1_\delta$ (it will be usually the positive real number~$\delta$)
and consider $\pi_1(\mathbb{D}^2_\delta\setminus\{0\};\hat\delta)$, isomorphic to $\mathbb{Z}$ admitting as canonical generator
a loop travelling the boundary counterclockwise. Then, the map
\begin{equation}\label{eq:localbraid}
\begin{tikzcd}[row sep=0pt]
\mathbb{D}^2_\delta\setminus\{0\}\rar["\hat{p}"]&T_n\\
x_0\rar[mapsto]&p(x_0,t)
\end{tikzcd}
\end{equation}
induces a morphism $\hat{p}_*:\pi_1(\mathbb{D}^2_\delta\setminus\{0\};\hat\delta)\to\pi_1(T_n;\hat{p}(\hat\delta))$; the target
is a braid group. We are going to relate this map with actual braids using Puiseux expansions. Under canonical identifications
this map will be independent of $\varepsilon,\delta$.

Let us introduce some notations. Let $C_{\varepsilon,\delta}$ be the representative of $C$ 
in ${\mathbb{D}}^2_\delta\times\mathbb{D}^2_\varepsilon$
and let
$C_{\varepsilon,\delta}^*:=C_{\varepsilon,\delta}\setminus\{(0,0)\}$.

\begin{lema}[Puiseux expansion]\label{lem:puiseux}
The space $C_{\varepsilon,\delta}^*$ is the restriction of a smooth analytic subvariety of 
an open neighbourhood of 
$\left(\mathbb{D}^2_\delta\times \mathring{\mathbb{D}}^2_\varepsilon\right)\setminus\{(0,0)\}$
and the projection  $C_{\varepsilon,\delta}^*\to\mathbb{D}^2_\delta$
is the restriction of an unramified holomorphic cover. 

In particular, if $C_1^*\cup\dots\cup C_r^*$ is the decomposition in connected components 
of $C_{\varepsilon,\delta}^*$, then for each $j\in\{1,\dots,r\}$ there is an isomorphism
\[
\begin{tikzcd}[row sep=0,/tikz/column 1/.append style={anchor=base east},/tikz/column 2/.append style={anchor=base west}]
\mathbb{D}^2_{\delta_j}\setminus\{0\}\rar&C_j^*\\
t\rar[mapsto]&(t^{m_j},h_j(t))
\end{tikzcd}
\]
where $\delta_j^{m_j}=\delta$.
\end{lema}

\begin{proof}
Let $(x_0,y_0)\in C_{\varepsilon,\delta}^*$. We have that $p(x_0, y_0)=0$ and $\disc_y p(x_0)\neq 0$, hence
$\frac{\partial p}{\partial y}(x_0,y_0)\neq 0$, and from the Implicit Function Theorem 
$C_{\varepsilon,\delta}^*$ is the restriction of a smooth subvariety of a neighbourhood of $\left(\mathbb{D}^2_\delta\times\mathring{\mathbb{D}}^2_\varepsilon\right)\setminus\{(0,0)\}$
and the map
\[
\begin{tikzcd}[row sep=0,/tikz/column 1/.append style={anchor=base east},/tikz/column 2/.append style={anchor=base west}]
\mathring{C}_{\varepsilon,\delta}^*\rar["\pi"]&\mathbb{D}^2_\delta\setminus\{0\}\\
(x,y)\rar[mapsto]&x
\end{tikzcd}
\]
is the restriction of a local analytic diffeomorphism. Since all the fibers have the same cardinality~$n$, it is an unramified cover, maybe non connected.
Let $C_i^*$ be a connected component of $C_{\varepsilon,\delta}^*$; the restriction of $\pi$ is still a covering,
and it must be a cyclic cover of $\mathbb{D}^2_\delta\setminus\{0\}$. Let $m_j$ be the degree of this covering
and let $\delta_j$ be the $m_j^{\text{th}}$-root of $\delta$. Any two such covers are isomorphic, so the following
diagram holds:
\[
\begin{tikzcd}
C_j^*
\ar[rd, "\pi"]
&&
\ar[ll, dashrightarrow, "\Phi_j" above]
\ar[ld, "\varphi_{m_j}" above=3pt]
\mathbb{D}^2_{\delta_j}\setminus\{0\}
\\
&
\mathbb{D}^2_\delta\setminus\{0\}
&
\end{tikzcd}
\]
where $\varphi_{m_j}(t)=t^{m_j}$. Hence,
there is a holomorphic function whose restriction is $h_j:\mathbb{D}^2_{\delta_j}\setminus\{0\}\to\mathring{\mathbb{D}}^2_\varepsilon$
such that $\Phi_j(t)=(t^{m_j},h_j(t))$. Since $h_j$ is locally bounded it can be holomorphically extended to $0$
with $h_j(0)=0$.
\end{proof}

\begin{remark}
The extended maps $\Phi_j:\mathbb{D}^2_\delta\to\mathbb{C}^2$ 
are the \emph{Puiseux expansions} of~$C_{\varepsilon,\delta}$.
\end{remark}

\begin{lema}\label{lem:puiseux_braid}
For $j=1,\dots,r$ and $k=1,\dots,m_j$, let us consider the paths
\[
\begin{tikzcd}[row sep=0,/tikz/column 1/.append style={anchor=base east},/tikz/column 2/.append style={anchor=base west}]
{[0,1]}\rar["\gamma_{j,k}"]&\bc\\
s\rar[mapsto]&h_j\left(\hat\delta^{\frac{1}{m_j}}\exp\left(2\sqrt{-1}\pi\frac{s+k-1}{m_j}\right)\right).
\end{tikzcd}
\]
The origins and ends of these paths are the roots of $p(\hat{\delta},t)$ and they form the braid 
image by $\hat{p}_*$ of the canonical generator of $\pi_1(\mathbb{D}^2_\delta\setminus\{0\};\hat\delta)$.
\end{lema}

The proof of this lemma is straightforward. Let us see some examples.

\begin{ejm}\label{ex:onepair}
Let us consider $f=p=y^n-x^m$. Since it is quasihomogeneous, we may take $\varepsilon=2$ and $\delta=1$. We take $\hat\delta=1$.
Let $d:=\gcd(n,m)$ and let $n_1=\frac{n}{d}$, $m_1=\frac{m}{d}$. The Puiseux expansions of 
Lemma~\ref{lem:puiseux} are 
\[
\begin{tikzcd}
t\rar[mapsto]&(t^{n_1},\exp\frac{2\sqrt{-1}\pi\ell}{d}\,t^{m_1}),\quad
\ell\in\{0,1,\dots,d-1\}.
\end{tikzcd}
\]
Note that the roots of $p(1,y)$ are the $n$-roots of unity. In order to identify the braid group with $\mathbb{B}_n$
we choose the path as in the right-hand side of Figure~\ref{fig:path}, i.e, an arc in the circle of radius~$1$ centered at
the origin. It is left to the reader to check that with this identification with $\mathbb{B}_n$ the braid obtained is
$(\sigma_{n-1}\cdot\ldots\cdot\sigma_1)^m$. In particular, for $n=2$, we obtain $\sigma_1^m$
and for $m=1$ we obtain $\sigma_{n-1}\cdot\ldots\cdot\sigma_1$.
\end{ejm}

It is customary to express this previous series as $h_j\left(x^{\frac{1}{m_j}}\right)$, i.e., a power series with rational exponents,
with bounded denominators. Note that this series is not uniquely defined, since 
$h_j\left(\zeta_j x^{\frac{1}{m_j}}\right)$, for any $m_j$-root of unity, plays the same role. In the language of Lemma~\ref{lem:puiseux_braid},
there is a cyclic permutation of $\gamma_{j,1},\dots,\gamma_{j,m_j}$.
Let us consider 
\[
p_j(x,y):=\prod_{\zeta_j^{m_j}=1} \left(y-h_j\left(\zeta_j x^{\frac{1}{m_j}}\right)\right).
\]
It is not hard to prove that $p_j\in\bc\{x\}[y]$ (irreducible), $p=p_1\cdot\ldots\cdot p_r$, and that $C_j:=C_j^*\cup\{0\}$,
is a representative of $p_j^{-1}(0)$. Following the same ideas of the previous lemma, we determine 
a braid in $\mathbb{B}_{m_j}$ whose closure is a knot.

Given $h_j$, we can write it down as
\[
h_j\left(x^{\frac{1}{m_j}}\right)=
h_{j,0}(x)+
\sum_{\ell=1}^g
x^{\frac{p_{j,\ell}}{q_{j, 1}\cdot\dots\cdot q_{j, \ell}}}
h_{j,g}\left(x^{\frac{1}{q_{j, 1}\cdot\dots\cdot q_{j, \ell}}}\right)
\]
where $q_1,\dots,q_g>1$ and
\[
\frac{p_{j,1}}{q_{j, 1}}<\dots<\frac{p_{j,\ell}}{q_{j, 1}\cdot\dots\cdot q_{j, \ell}}<\dots<
\frac{p_{j,g}}{q_{j, 1}\cdot\dots\cdot q_{j, g}}.
\]
If we assume that $h_{j,0}$ contains all the terms till 
the term with exponent $\frac{p_{j,\ell}}{q_{j, 1}}$
and a similar condition for the next terms, then
this decomposition is unique.

Let us set 
\begin{equation}\label{eq:puiseux_exp}
\frac{p_{j,\ell}}{q_{j, 1}\cdot\dots\cdot q_{j, \ell}}=:\frac{c_{j,\ell}}{m_j}
\end{equation}
\begin{prop}\label{prop:isotopy1}
The braids determined by $h_j$ and by 
\[
\tilde{h}_j(x)=\sum_{\ell=1}^g x^{\frac{c_{j,\ell}}{m_j}}
\]
are conjugate. In particular, the knots determined by $p_j$ and 
\[
\tilde{p}_j(x,y):=\prod_{\zeta_j^{m_j}=1} \left(y-\tilde{h}_j\left(\zeta_j x^{\frac{1}{m_j}}\right)\right)
\]
are isotopic.
\end{prop}

\begin{proof}[Idea of the proof]
If $m_j=1$, there is nothing to prove since the braids have only one strand. If $m_j>1$, then 
$h_{j,0}(x)$ is a polynomial and the map
\[
\begin{tikzcd}[row sep=0,/tikz/column 1/.append style={anchor=base east},/tikz/column 2/.append style={anchor=base west}]
\mathbb{C}\times{[0,1]}\rar&\mathbb{C}\times{[0,1]}\\
(y, t)\rar[mapsto]& 
\left(\dfrac{y - h_{j,0}\left(\hat{\delta}^{\frac{1}{m_j}}\frac{2\sqrt{-1}\pi t}{m_j}\right)}{h_{j,1}\left(\hat{\delta}\exp\left(\hat{\delta}^{\frac{1}{m_j}}\frac{2\sqrt{-1}\pi t}{m_j}\right)\right)}, t\right)
\end{tikzcd}
\]
sends the braid $\tau_{\tilde{h}_j}$ to a braid whose end points are close 
to the ones of $\tau_{h_j}$, and this isotopy induces a conjugation of the braids. 
The end points for both braids
are distributed in $q_{j,1}$ disks
centered at the $q_{j,1}$-roots of $\hat{\delta}^{p_{j,1}}$.
Actually both braids are inside $q_{j,1}$ cylinders 
which are neighbourhoods of the strands of the braid $\tau_1$ associated to
$y^{q_{j,1}}-x^{p_{j,1}}$. If $g=1$, actually $\tau_1=\tau_{\tilde{h}_j}$
and since each strand of $\tau_{h_j}$ is in one of the cylinders,
then the two braids are isotopic.

For general $g>1$, we iterate this process, but now the isotopies are performed
inside these cylinders for the first step, and inside narrower (and longer) cylinders
for the next $g-1$ steps.
\end{proof}

These ideas are borrowed from~\cite[Theorem~3]{Brieskorn1986}, which contains
more details. A more detail explanation can be found in \cite{pham}
and \cite{Michel-Weber}.

\begin{remark}
Why do those isotopies induce conjugation and not equality of braids? If we have
two braid representatives with the same end points and we have an isotopy
which fixes the base points then the two braids are indeed equal. If we allow the isotopy
to move the end points (the same motion on both sides), 
actually the motion of these base points defines a conjugating braid
between the two braids. 
In Proposition~\ref{prop:isotopy1}, the two braids have two distinct systems
of end points. In \S\ref{sec:braid}, we identified these braids with elements
in $\mathbb{B}_n$, but the identification was made using a specific motion
from the end points of a braid to $\{1,\dots,n\}$. The motions defined by
an actual isotopy can be different, and this is why in general one cannot assume
equality in Proposition~\ref{prop:isotopy1}.
\end{remark}

For series $h$ in $x$ with either integer or rational exponents, we define 
the $x$-order $\nu_x(h)$ of $h$ as $\infty$, if $h=0$, or as the exponent of the first 
non-zero term, otherwise.
Note for example that for 
\[
\left\{\nu_x\left(h_j\left(\zeta_j x^{\frac{1}{m_j}}\right)-h_j\left(x^{\frac{1}{m_j}}\right)\right)
\,\middle|\,\zeta_j^{m_j}=1, \zeta_j=1
\right\}
\]
is exactly the set exponents in \eqref{eq:puiseux_exp},
the \emph{characteristic Puiseux exponents}.

Something similar can be done when we have several branches.
Given two different branches $h_j, h_k$, let
\[
\mu_{j, k}:=\max
\left\{\nu_x\left(h_j\left(\zeta_j x^{\frac{1}{m_j}}\right)-h_k\left(\zeta_k x^{\frac{1}{m_k}}\right)\right)
\,\middle|\,\zeta_j^{m_j}=\zeta_k^{m_k}=1
\right\}.
\]
Let us denote 
\[
h_j(x) = \sum_{\ell=1}^\infty a_{\ell,j} x^{\frac{\ell}{m_j}}.
\]
Let $\mathcal{S}_j\subset\mathbb{N}$ be the subset formed by the exponents $\frac{c_{j,\ell}}{m_j}$
and the exponents $\mu_{j,k}$ for $k\neq j$ (known as \emph{coincidence exponents}).
Let 
\[
\hat{h}_j(x):=\sum_{\frac{\ell}{m_j}\in\mathcal{S}_j} a_{\ell,j} x^{\frac{\ell}{m_j}}.
\]

This proposition follows the same ideas as Proposition~\ref{prop:isotopy1},
see \cite[Theorem~15]{Brieskorn1986} and also \cite{pham, Michel-Weber}. 
Lê's \emph{carrousel}~\cite{Le:78,Le:79} is guided by the same ideas.

\begin{prop}
The braids determined by $\{h_j\}_{j=1}^r$ and by $\{\hat{h}_j\}_{j=1}^r$ 
are conjugate. In particular, the links determined by $p$ and 
\[
\prod_{j=1}^r\left(\prod_{\zeta_j^{m_j}=1} \left(y-\hat{h}_j\left(\zeta_j x^{\frac{1}{m_j}}\right)\right)\right)
\]
are isotopic.
\end{prop}

One can say more about these braids. Note that for the choice of $\Gamma$ in Example~\ref{ex:onepair}, the braid is \emph{positive},
i.e., it can be represented by a positive word in the Artin generators. Whenever there is more than one Puiseux exponent (or coincidence
exponent) an application of the \emph{L{\^e} carrousel} gives the hint to find the right $\Gamma$
such that the corresponding local braid is positive, connecting wisely the 
different levels of the carrousel.

\begin{prop}\label{prop:positive}
There exists a choice of $\Gamma$ such that the local braid associated to a germ of curve singularity is positive.
\end{prop}

\begin{remark}
It is an easy exercise to show that the degree of this braid is related to the geometry of the curve. Namely, this degree
equals the sum of its Milnor number and the local intersection with $\{x=0\}$ minus one. 
\end{remark}

 \section{Semilocal braid monodromy}\label{sec:semilocal}

Let us consider $f(x,y)\in\bc\{x\}[y]$ a reduced monic polynomial of degree~$n>0$, but we do not assume it is a Weiestra{\ss} polynomial. Following Hensel's Lemma, if
$f(0,y)=\prod_{j=1}^r (y+y_j)^{m_j}$ is the factorization, then 
\[
f(x,y)=\prod_{j=1}^r f_j(x,y+y_j)
\]
where $f_j(x,y)$ is a Weiestra{\ss} polynomial of degree~$m_j$. Following the ideas in \S\ref{sec:local}, we can find $\varepsilon_1,\dots,\varepsilon_r>0$ and $\delta>0$
such that the disks $\mathbb{D}^2_{\varepsilon_j}(y_j)$ are pairwise disjoint and the pairs $(\varepsilon_j,\delta)$ satisfy the conditions in Proposition-Definition~\ref{propdef:ed}.

Fix $\hat{\delta}$ as in \S\ref{sec:local}. The set of (pairwise distinct) roots of $f(\hat\delta,y)$ is the disjoint union of the sets 
of roots of $f_j(\hat\delta,y+y_j)$. As in Lemma~\ref{lem:puiseux_braid}, the choice of a path $\Gamma^j$ allows to construct
a braid $\tau_j\in\mathbb{B}_{m_j}$. Since the roots of $f_j(x,y+y_j)$ are in pairwise disjoint disks, we can construct a global 
path $\Gamma$ joining $\Gamma^j$ with $\Gamma^{j+1}$. This new path defines
a homomorphism
\begin{equation}\label{eq:prod_conm}
\prod_{j=1}^r \mathbb{B}_{m_j}\to\mathbb{B}_n.
\end{equation}
The images of each factor pairwise commute. As in \eqref{eq:localbraid} we can consider
\begin{equation}\label{eq:semilocalbraid}
\begin{tikzcd}[row sep=0pt]
\mathbb{D}^2_\delta\setminus\{0\}\rar["\hat{f}"]&T_n\\
x_0\rar[mapsto]&f(x_0,t)
\end{tikzcd}
\end{equation}
which induces a map $\hat{f}_*:\pi_1(\mathbb{D}^2_\delta\setminus\{0\};\hat{\delta})\to\pi_1(T_n;\hat{p}(\hat{\delta}))\equiv\mathbb{B}_n$. The image of 
a counterclockwise generator is obtained using the morphisms associated to each Weiestra{\ss} polynomials from the morphisms in  \eqref{eq:localbraid}
and the morphism~\eqref{eq:prod_conm}.
 \section{Global braid monodromy}\label{sec:global}

We are going to connect the previous results to recall the definition of braid monodromy. 
Let us fix $f(x,y)\in\bc[x,y]$ a polynomial of degree~$n>0$ in $y$ which is also monic in~$y$, i.e.,
\[
f(x,y)=y^n + \sum_{j=1}^{n} f_j(x) y^{n-j};
\]
since we are interested only in the zero locus $C:=f^{-1}(0)$, we assume that $f$ is reduced.
In particular $\disc_y(f)(x)\in\bc[x]$ is a non-zero polynomial and 
\[
\Delta_f:=\{t\in\bc\mid\disc_y(f)(t)=0\}
\]
is a finite set with cardinality $r\geq 0$. Moreover, we have a continuous map
\[
\begin{tikzcd}[row sep=0,/tikz/column 1/.append style={anchor=base east},/tikz/column 2/.append style={anchor=base west}]
\bc\setminus\Delta_f\rar["\tilde{f}"]&T_n\equiv Y_n\\
t\rar[mapsto]&f(t,y).
\end{tikzcd}
\]
We fix an element $x_0\in\bc\setminus\Delta_f$.
\begin{dfn}
A \emph{braid monodromy} of $f$ is a map
\[
\tilde{f}_*:\pi_1(\bc\setminus\Delta_f;x_0)\to\pi_1(T_n;\tilde{f}(x_0)).
\]
\end{dfn}

\begin{remark}
The condition of $f$ being monic is not necessary.
Braid monodromies can be defined even when the leading coefficient in~$y$ has positive degree in $x$, i.e.,
there are vertical asymptotes. How to use braid monodromy to obtain the fundamental group of the complement
is much more involved in the presence of vertical asymptotes.
\end{remark}

A first version of the Zariski-van Kampen method is the following theorem. We postpone its proof until we have more 
features of braid monodromy.

\begin{thm}[First version of Zariski-van Kampen]
Let $y_0\in\bc$ such that the point $(x_0,y_0)\notin C$.
Then $\pi_1(\bc^2\setminus C;(x_0,y_0))$ is isomorphic to the quotient of
$\pi_1(\bc\setminus\tilde{f}(x_0);y_0)$ (where $\tilde{f}(x_0)$ is seen as an element of $Y_n$)
by the subgroup normally generated by 
\[
\{\mu^\tau\cdot\mu^{-1}\mid\mu\in\pi_1(\bc\setminus\tilde{f}(x_0);y_0), \tau\in\pi_1(Y_n;\tilde{f}(x_0))\}.
\]
\end{thm}
Since all the above groups are finitely presented, it is clear that 
$\pi_1(\bc^2\setminus C;(x_0,y_0))$ has a finite presentation.

\begin{dfn}
A \emph{braid monodromy factorization} of  a polynomial $f$ is given by a 
tuple $(\tilde{f}_*(\gamma_1),\dots,\tilde{f}_*(\gamma_r))\in(\mathbb{B}_n)^r$  where $(\gamma_1,\dots,\gamma_r)$ is a pseudogeometric basis of $\pi_1(\bc\setminus\Delta_f;x_0)$.
\end{dfn}
It is clear that a braid monodromy factorization determines also $\pi_1(\bc\setminus\tilde{f}(x_0);y_0)$. 
Let us see how many braid monodromy factorizations $f$ has.

\begin{dfn}
The \emph{Hurwitz action} of $\mathbb{B}_n\times\mathbb{B}_r$ on $(\mathbb{B}_n)^r$ is defined as follows. 
If $\tau\in\mathbb{B}_n$, then
\[
(\tau_1,\dots,\tau_r)^\tau:=(\tau_1^\tau,\dots,\tau_r^\tau).
\]
Let $\eta_i$ be an Artin generator of $\mathbb{B}_r$. Then
\[
(\tau_1,\dots,\tau_r)^{\eta_i}:=(\tau_1,\dots,\tau_{i-1},\tau_i\cdot\tau_{i+1}\cdot\tau_i^{-1},\tau_{i},\tau_{i+2},\dots,\tau_r).
\]
\end{dfn}

Given an element of $(\mathbb{B}_n)^r$ there are a couple of linked invariants which are simpler.

\begin{dfn}
Let $\boldsymbol{\tau}:=(\tau_1,\dots,\tau_r)\in(\mathbb{B}_n)^r$. The \emph{monodromy group}
$G(\boldsymbol{\tau})$ is the subgroup of $\mathbb{B}_n$ generated by the coordinates of~$\boldsymbol{\tau}$.
The \emph{pseudo-Coxeter} element of~$\boldsymbol{\tau}$ is $\mathbf{e}(\boldsymbol{\tau}):=\tau_1\cdot\ldots\cdot\tau_r$. 
\end{dfn}

These two elements are \emph{almost} fixed by Hurwitz moves. The following result is straightforward.

\begin{lema}
Let $\boldsymbol{\tau}\in(\mathbb{B}_n)^r$. The $\mathbb{B}_n$-conjugacy class of the pair $(G(\boldsymbol{\tau}),\mathbf{e}(\boldsymbol{\tau}))$ is an invariant of the orbit of $\boldsymbol{\tau}$.
\end{lema}

A direct consequence of Theorem~\ref{thm:artin_action} is the following description of braid monodromy
factorizations.

\begin{lema}
The braid monodromy factorizations of $f$ form an orbit of the action  of 
$\mathbb{B}_n\times\mathbb{B}_r$ on $(\mathbb{B}_n)^r$. If 
$\boldsymbol{\tau}\in(\mathbb{B}_n)^r$ is a braid monodromy factorization of~$f$, then 
the isomorphism type of $\pi_1(\bc^2\setminus C;(x_0,y_0))$ is determined by the
conjugacy class of $G(\boldsymbol{\tau})$.
\end{lema}

The first version of Zariski-van Kampen Theorem is a direct consequence of this second version.

\begin{thm}[Second version of Zariski-van Kampen]\label{thm:zvk2}
Let $(\tau_1,\dots,\tau_r)\in(\mathbb{B}_n)^r$ be a braid monodromy factorization and 
let $\mu_1,\dots,\mu_n$ be a basis of the free group. Then
$\pi_1(\bc^2\setminus C;(x_0,y_0))$ is generated by $\mu_1,\dots,\mu_n$ 
with the relations 
\[
\{\mu_i^{\tau_j}\cdot\mu_i^{-1}\mid 1\leq j\leq r, 1\leq i<n\}.
\]
\end{thm}
Since $\mu_1\cdot\ldots\cdot\mu_n$ is fixed by the braid action, we do not need the relations
for $\mu_n$.

\begin{proof}[Idea of the proof]
Let $C_\varphi$ be the union of $C$ and the vertical lines $x=x'$ with $x'\in\Delta_f$.
The vertical projection $\pi:\bc^2\setminus C_\varphi\to\bc\setminus\Delta_f$ is a locally trivial fibration
with fiber homeomorphic to $\bc\setminus\{n\text{ points}\}$. 

To prove this, we consider $\bc^2=\bp^2\setminus L_\infty$, and let $P$ be the 
point in the line at infinity where all the vertical lines meet.
Let $\Sigma$ be the blow-up of this point, and let $\pi:\Sigma\to\bp^1\equiv\bc\cup\{\infty\}$ the completion of the vertical projection.
Let $E$ be the exceptional component and let $V:=\pi^{-1}(\bc\setminus\Delta_f)$.

Let us consider the pair $(V,(C\cup E)\cap V)$. The restriction of $\pi$ to each element of the pair is proper. The restriction $\pi_{|V}$
is a trivial fibration and the restriction to $(C\cup E)\cap V$ is a local homeomorphism where the fibers
have finite and constant cardinality.
Hence, it is a covering, i.e., a locally trivial fibration with discrete fibers.
We have seen that $\pi:(V,(C\cup E)\cap V)\to\bc\setminus\Delta_f$
is  a locally trivial fibration of pairs. As a consequence, for the complement we also have 
that $\pi:\bc^2\setminus C_\varphi\to\bc\setminus\Delta_f$ is a locally trivial fibration
with fiber a $d$-punctured affine plane.

Choose a base point $(x_0,y_0)$ with $\abs{x_0},\abs{y_0}\gg 1$.
Since the base has the homotopy type of a graph, the last terms of the long exact sequence of homotopy are:
\[
\begin{tikzcd}[column sep=.7cm]
1\rar[]&\pi_1(\bc\setminus\tilde{f}(x_0);y_0)\rar[]&\pi_1(\bc^2\setminus C_\varphi;(x_0,y_0))\rar[]&\pi_1(\bc\setminus\Delta_f;x_0)\rar[]&1.
\end{tikzcd}
\]
Since the last group is free, this sequence splits and actually this is done using the inclusion in the line $y=y_0$ (using compacity, this works
as long as $\abs{y_0}$ is big enough). The lifting of the meridians in $\pi_1(\bc\setminus\Delta_f;x_0)$ are meridians of the vertical lines.
We can compute
its fundamental group and then kill the meridians of the vertical lines, see~\cite[Lemma~4.18]{Fujita1982}.
\end{proof}

In order to be closer to Zariski's ideas, let us define a new concept.
Let us fix a braid monodromy factorization $(\tau_1,\dots,\tau_r)$.
Given $\tau_i$, as an application of Proposition~\ref{prop:positive},
we can assume that $\tau_i=\alpha_i^{-1}\cdot\beta_i\cdot\alpha_i$, where 
$\beta_i$ is a positive braid. Since $\tau_i$ is the image of a meridian
this decomposition can be obtained geometrically from the decomposition of the
meridian, see Definition~\ref{def:meridian}; the identification
between braids with distinct end points and elements of $\mathbb{B}_n$ of 
\S\ref{sec:braid} provides the 
conjugation decomposition of~$\tau_i$. Then, using the ideas in \S\ref{sec:semilocal}, if there are
several singular points of the projection, then $\beta_i$ decomposes
as a product $\beta_{i,1}\cdot\ldots\cdot\beta_{i,s_i}$, $s_i\geq 1$; 
there are $n_{i, j}$ consecutive strands involved in $\beta_{i, j}$,
starting at the strand $t_{i,j}$,
with the condition that $n_{i,1}+\dots+n_{i,s_i}\leq n$.
The tuples of strands are pairwise disjoint and hence the braid factors 
pairwise commute.

\begin{remark}
Although we can give a particular decomposition of the braid~$\tau$ using the decomposition of a meridian
we can choose any decomposition, as it can be deduced from \cite{kt:00,Carmona2003}.
\end{remark}

\begin{dfn}\label{def:puiseux1}
The tuple $((\alpha_1,(\beta_{1,1},\dots,\beta_{1,s_1})),\dots,(\alpha_r,(\beta_{r,1},\dots,\beta_{r,s_r})))$
is a \emph{Puiseux factorization} of $f$.
\end{dfn}

The third version of the Zariski-van Kampen Theorem is the closest one to the original ideas.

\begin{thm}[Third version of Zariski-van Kampen]\label{thm:zvk3}
Consider a Puiseux factorization of~$f$ given by
$((\alpha_1,(\beta_{1,1},\dots,\beta_{1,s_1})),\dots,(\alpha_r,(\beta_{r,1},\dots,\beta_{r,s_r})))$.
Let $\mu_1\!,\!\dots\!,\!\mu_n$ be a basis of the free group. 
Then
$\pi_1(\bc^2\setminus C;(x_0,y_0))$ is generated by $\mu_1,\dots,\mu_n$ 
with the relations 
\begin{equation}\label{eq:puiseux}
\left\{\left(\mu_k^{\beta_{i,j}}\cdot\mu_k^{-1}\right)^{\alpha_i}\middle|\, 1\leq i\leq r,\ 1\leq j\leq s_i,\ 
t_{i,j}\leq k< t_{i,j} + n_{i,j} \right\}.
\end{equation}
\end{thm}

\begin{proof}
Let us denote by $\tau_j$ the $j$-factor of the braid monodromy factorization. Let us denote 
by $\ns{\bullet}$ the smallest normal subgroup generated by $\bullet$. Note that 
\begin{gather*}
\ns{\underset{1\leq i<n}{\mu_i^{\tau_j}\cdot\mu_i^{-1}}}=
\ns{\underset{\mu\in\mathbb{F}_n}{\mu^{\tau_j}\cdot\mu^{-1}}}
=
\ns{\underset{\mu\in\mathbb{F}_n}{\mu^{\alpha_j\cdot\tau_j}\cdot(\mu^{-1})^{\alpha_j}}}
\\
=
\ns{\underset{\mu\in\mathbb{F}_n}{\mu^{\left(\prod_{k=1}^{s_j}\beta_{j,k}\right)\cdot\alpha_j}\cdot(\mu^{-1})^{\alpha_j}}}=
\ns{\underset{1\leq i<n}{\left(\mu_i^{\prod_{k=1}^{s_j}\beta_{j,k}}\cdot\mu_i^{-1}\right)^{\alpha_j}}}.
\end{gather*}
We obtain the statement by the elimination of the trivial elements.
\end{proof}
This is not only more economic, but it is also part of a stronger theorem on the homotopy type due to Libgober.
Libgober's statement has some extra hypotheses about genericity but actually his proof
provides this more general statement.

\begin{thm}[\cite{Libgober1986}]\label{thm:lib}
Let $X$ be the $CW$-complex associated to the presentation of $\pi_1(\bc^2\setminus C)$ in Theorem{\rm~\ref{thm:zvk3}}.
Then, $X$ is homotopy equivalent to $\bc^2\setminus C$.
\end{thm}

\begin{remark}
Let us recall the interpretation of the relations in \eqref{eq:puiseux}. Let us express them first in terms of 
the positive braids $\beta_{i, j}$ (i.e. before the action of $\alpha_i$).
\begin{enumerate}[label=\rm(R\arabic{enumi})]
\item $\beta_{i,j}=\sigma_1^k$. The only relation is $\mu_1=\mu_2$ (for $k=1$), $[\mu_1,\mu_2]=1$ (for $k=2$), the braid
relation $\mu_1\cdot\mu_2\cdot\mu_1=\mu_2\cdot\mu_1\cdot\mu_2$, \dots
\item $\beta_{i, j}=\sigma_{m-1}\cdot\ldots\cdot\sigma_1$. The relations are $\mu_1=\dots=\mu_m$.
\item $\beta_{i, j}=(\sigma_{m-1}\cdot\ldots\cdot\sigma_1)^m$. The relations are the commutations of $\mu_1\cdot\ldots\cdot\mu_m$
with $\mu_1,\dots,\mu_{m-1}$.
\end{enumerate}
The action of $\alpha_i^{-1}$ implies that these relations may apply to some other geometric basis of meridians instead of $\mu_1,\dots,\mu_n$.
\end{remark}

\begin{remark}\label{rem:non_gen1}
In the next section we are going to impose some genericity conditions which are enough for most results. The reason to state them in more generality
comes from the fact that these non-generic braid monodromy factorizations
can be computed usually in an easier way, sometimes it may make the difference between effective computation and failure. 
Even more, if the failure of $f(x,y)$ to be monic in~$y$ is because $f(x,y)=f_1(x,y)\cdot g(x)$, where $f_1$ is monic in~$y$, then, with the braid monodromy of $f_1$, the fundamental group of $\bc^2\setminus\{f(x,y)=0\}$ can be computed,
using the ideas in the proof of Theorem~\ref{thm:zvk2}.
\end{remark}

 \section{Topology of plane projective curves}\label{sec:proj}

Zariski's original goal was to obtain properties of the topology of projective plane curves but in \S\ref{sec:global} only the affine case was considered. In that section the coordinate system and the equation of the affine curve was fixed. Let $\bar{C}\subset\bp^2$ be a projective plane
curve of degree~$n$. Let $L\subset\bp^2$ be a line transversal to $\bar{C}$, i.e., $\#(L\cap \bar{C})=n$; note that there is a dense open Zariski set 
in the space of lines satisfying this property. Choose a point $P\in L$ in a generic way. We ask for the following genericity properties:
\begin{enumerate}[label=\rm(G\arabic{enumi})]
\item $P\notin\bar{C}$.
\item Let $M\neq L$ be a line through $P$ which is tangent to $\bar{C}$ at a point $Q$, then this point is unique and $(M\cdot\bar{C})_Q=2$, i.e., it is not a flex point, and 
$M\cap\sing\bar{C}=\emptyset$.
\item Let $M\neq L$ be a line through $P$  such that $M\cap\sing\bar{C}\neq\emptyset$; then,  this set contains only a point $Q$, and $(M\cdot\bar{C})_Q=m_Q$, 
where $m_Q$ is the multiplicity of $\bar{C}$ at $Q$.
\end{enumerate}

It is a classical result that in the incidence variety $\{(L,P)\in\check{\bp}^2\times\bp^2\mid P\in L\}$, the space of 
pairs satisfying the above properties is a connected dense open Zariski subset. We need to eliminate the lines joining
two singular points, the tangent lines to flexes, the bitangent lines (or worse), the lines non-transversal to the singular points, and
the tangent lines through the singular points. This is a finite number of lines, take $P$ in the complement of $\bar{C}$ and these lines. Now, take $L$
through $P$ and transversal to $\bar{C}$. 

The following result is well known though I do not have a concrete reference (maybe~\cite{Randell1989} for line arrangements). Let $\bp_{(n)}$ be the projective space (of dimension~$\frac{n(n+3)}{2}$)
of the projective plane curves of degree~$n$.
Let $\Sigma\subset\bp_{(n)}$ be the space of curves with the same \emph{combinatorics}. For an irreducible curve, its combinatorics is the degree and the 
topological types of the singular points. For a reducible curve, the combinatorics is determined by the combinatorics of the irreducible components
and the topological types and intersection numbers at the intersection of components.
These spaces $\Sigma$ are quasi-projective varieties.

\begin{prop}\label{prop:isotopy}
Let $\bar{C}_0,\bar{C}_1$ be two curves with the same combinatorics in the same connected component of $\Sigma$.

Let $(L_0,P_0), (L_1,P_1)$ with the above genericity properties for $\bar{C}_0,\bar{C}_1$ respectively. Then, there is a \emph{continuous path} $(\bar{C}_t,L_t,P_t)$, $t\in[0,1]$
such that
\begin{enumerate}[label=\rm(\alph{enumi})]
\item $\bar{C}_t\in\Sigma$;
\item $(L_t,P_t)\in\check{\bp}^2\times\bp^2$ satisfying the genericity properties for $\bar{C}_t$,
\item and there exists an isotopy $H:\bp^2\times[0,1]\to\bp^2$, $h(Q,t):=h_t(Q)$, such that 
\begin{enumerate}[label=\rm(\roman{enumii})]
\item $h_0$ is the identity, 
\item $h_t(\bar{C}_0)=\bar{C}_t$, 
\item $h_t(L_0)=L_t$, and 
\item $h_t(P_0)=P_t$.
\end{enumerate}
\end{enumerate}

\end{prop}

Given the curve $\bar{C}$ consider a system of coordinates $[X:Y:Z]$ in $\bp^2$ such that $L=\{Z=0\}$ and $P=[0:1:0]$. 
The affine chart $\bp^2\setminus L$ is identified with $\bc^2$ via the map $(x, y)\mapsto[x:y:1]$. Let $F(X,Y,Z)=0$ be an equation
of $\bar{C}$, where $F$ is a homogeneous polynomial of degree~$d$. Since $P\notin\bar{C}$, the coefficient of $Y^n$ does not vanish
and we may assume it equals~$1$. Let $f(x,y):=F(x,y,1)$, which is a monic polynomial in~$y$ (as demanded in \S\ref{sec:global}). It is
the equation of $C:=\bar{C}\cap\bc^2$.

\begin{dfn}
A \emph{braid monodromy factorization} of $\bar{C}$ is a braid monodromy factorization for~$f$ with the above choices.
\end{dfn}

\begin{prop}
Let $\boldsymbol{\tau}:=(\tau_1,\dots,\tau_r)\in(\mathbb{B}_n)^r$ be a braid monodromy factorization of~$\bar{C}$.

\begin{enumerate}[label=\rm(\arabic{enumi})]
\item\label{r1} $r=\deg\check{\bar{C}}+\#\sing(\bar{C})$ where $\check{\bar{C}}$ is the dual curve of $\bar{C}$.
\item\label{r2} The embedded topology of $(\bc^2,C)$ does not depend on the
particular choice of $L,P$ and the coordinate system.
\item\label{r3} $\mathbf{e}(\boldsymbol{\tau})=\Delta_n^2$ is the \emph{full-twist} or positive generator of the center of $\mathbb{B}_n$.
\item The set of braid monodromy factorizations of $\bar{C}$ is the orbit of $\boldsymbol{\tau}$ by the Hurwitz action.
\item The $\mathbb{B}_n$-conjugacy class of $G(\boldsymbol{\tau})$ is an invariant of $\bar{C}$.
\end{enumerate}
\end{prop}

\begin{proof}
\ref{r1} is a direct consequence of the genericity properties and the definition of the dual curve.
\ref{r2} follows from Proposition~\ref{prop:isotopy}. \ref{r3} is a consequence of the genericity of $L$ and the other results
are straightforward.
\end{proof}

We can also define Puiseux factorizations in the same way, but now it becomes easier, since in the notation before 
Definition~\ref{def:puiseux1}, $s_{i}$ is always one; we simplify the notations, where $n_i$ is the number of strands
involved in the central positive braid, and $t_i$ is the index of the first strand. Let us rephrase the definitions and results.

\begin{dfn}
A \emph{Puiseux factorization} of $\bar{C}$ is a tuple $\{(\alpha_i,\beta_i)\}_{i=1}^r$ such that the tuple
$\{(\alpha_i\cdot\beta_i\cdot\alpha_i^{-1})\}_{i=1}^r$ is a braid monodromy factorization of $\bar{C}$ 
and the braids $\beta_i$ are positive.
\end{dfn}

The projective version is obtained using again \cite{Fujita1982} and killing the meridian of $L$
which turns out to be $(\mu_1\cdot\ldots\cdot\mu_n)^{-1}$.

\begin{thm}[Projective versions of Zariski-van Kampen]
Let $\bar{C}$ be a projective curve.
\begin{enumerate}[label=\rm(ZvK\arabic{enumi})]
\item Let $\boldsymbol{\tau}:=(\tau_1,\dots,\tau_r)\in(\mathbb{B}_n)^r$ be a braid monodromy factorization
of $\bar{C}$. Then,
\begin{enumerate}[label=\rm(\alph{enumii})]
\item $\pi_1(\bc^2\setminus C)=
\left\langle
\mu_1,\dots,\mu_n
\middle|
\mu_i^{\tau_j}\cdot\mu_i^{-1},\ 1\leq j\leq r,\ 1\leq i<n
\right\rangle$, or equivalently, 
the quotient of $\mathbb{F}_n$ by the normal subgroup generated by 
$\mu^\tau\cdot\mu^{-1}$, for $\mu\in\mathbb{F}_n$ and $\tau\in G(\boldsymbol{\tau})$.

\item $\mu_1\cdot\ldots\cdot\mu_n$ is a central element of $\pi_1(\bc^2\setminus C)$.

\item Actually, we can take the relations only for $1\leq j<r$ and add the centrality
of $\mu_1\cdot\ldots\cdot\mu_n$ as $n-1$ relations.

\item $\pi_1(\bp^2\setminus\bar{C})$ is the quotient of 
$\pi_1(\bc^2\setminus C)$ by the subgroup
 generated by $\mu_1\cdot\ldots\cdot\mu_n$.
\end{enumerate}
\item Let $\{(\alpha_i,\beta_i)\}_{i=1}^r$ be a Puiseux factorization of $\bar{C}$.
Then,
\begin{enumerate}[label=\rm(\alph{enumii})]
\item $\pi_1(\bc^2\setminus C)=
\left\langle
\mu_1,\dots,\mu_n
\middle|
(\mu_i^{\beta_j}\cdot\mu_i^{-1})^{\alpha_i},\ 1\leq j\leq r,\ 1\leq i<n
\right\rangle$.

\item The $CW$-complex of this presentation has the homotopy type
of $\bc^2\setminus C$.

\item $\pi_1(\bp^2\setminus\bar{C})=
\left\langle
\mu_1,\dots,\mu_n
\middle|
\underset{1\leq j< r}{\overset{1\leq i<n}{(\mu_i^{\ \beta_j}\cdot\mu_i^{-1})^{\alpha_i}\ }},\
\mu_1\cdot\ldots\cdot\mu_n
\right\rangle$.
\end{enumerate}

\end{enumerate}

\end{thm}

\begin{cor}
The groups $\pi_1(\bp^2\setminus\bar{C})$
and $\pi_1(\bc^2\setminus C)$ are completely determined by (the conjugacy class of) $G(\boldsymbol{\tau})$,
where $\boldsymbol{\tau}$ is any braid monodromy factorization of~$\bar{C}$.
\end{cor}

This corollary justifies the following definition.

\begin{dfn}
Let $\boldsymbol{\tau}\in(\mathbb{B}_n)^r$. The \emph{affine group} of $\boldsymbol{\tau}$ is
\[
G^\aff(\boldsymbol{\tau}):=\left\langle
\mu\in\mathbb{F}_n
\middle|\mu^\tau\cdot\mu^{-1},\ \tau\in G(\boldsymbol{\tau})\right\rangle.
\]
The \emph{projective group} of $\boldsymbol{\tau}$ is
\[
G^\proy(\boldsymbol{\tau}):=\frac{G^\aff(\boldsymbol{\tau})}{\ns{\mu_1\cdot\ldots\cdot\mu_n}}.
\]
\end{dfn}

It is useful to make explicit the following straightforward result.

\begin{prop}\label{prop:grupos}
Let $\boldsymbol{\tau}_1, \boldsymbol{\tau}_2\in(\mathbb{B}_n)^r$. If
$\boldsymbol{\tau}_1$ and $\boldsymbol{\tau}_2$ are Hurwitz equivalent,
$G^\aff(\boldsymbol{\tau}_1)\cong G^\aff(\boldsymbol{\tau}_2)$
and $G^\proy(\boldsymbol{\tau}_1)\cong G^\proy(\boldsymbol{\tau}_2)$.
\end{prop}

There is a close relationship between $\pi_1(\bc^2\setminus C)$ and $\pi_1(\bp^2\setminus\bar{C})$. Actually, they fit in 
this diagram:
\begin{equation}
\label{eq:pb}
\begin{tikzcd}
\mu_i\dar[mapsto]&[-30pt]\pi_1(\bc^2\setminus C)\rar["i_*"]\dar["\varepsilon"]&\pi_1(\bp^2\setminus\bar{C})\dar["\varepsilon_n"]&[-45pt]\mu_i\dar[mapsto]\\
1&\mathbb{Z}\rar["\pi"]&\mathbb{Z}/n&[-30pt]1\bmod{n}
\end{tikzcd}
\end{equation}

\begin{prop}
The diagram \eqref{eq:pb} is a pull-back diagram and $\ker i_*$ is central. If $\bar{C}$ is irreducible both
$\varepsilon$ and $\varepsilon_n$ are the abelianization maps and then $\pi_1(\bp^2\setminus\bar{C})$
determines $\pi_1(\bc^2\setminus C)$.
\end{prop}

\begin{remark}
If $\bar{C}$ is irreducible and $\varepsilon^{-1}(n\mathbb{Z})\cap\ctr(\pi_1(\bc^2\setminus C))$ is cyclic,
then $\pi_1(\bp^2\setminus\bar{C})$ can be recovered from $\pi_1(\bc^2\setminus C)$. 
\end{remark}

\begin{remark}
One last word about the genericity of braid monodromy factorizations. In general, it will be possible to 
obtain a generic braid monodromy from a non-generic one, which are usually computed in an easier way.
On the other side, the consideration of non-generic braid monodromies produces in some cases finer
invariants of the topology.
\end{remark}
 \section{Applications}\label{sec:app}

Braid monodromy factorizations (in both affine and generic cases) determine the topology of a plane curve, including its embedding in the plane.
It is useful in order to detect Zariski pairs, since it is a main tool for the computation of the fundamental group; this invariant (and related ones as the Alexander polynomial) has been widely used to find Zariski pairs, starting from the original example of Zariski (and many others).
But there are results showing that it is finer, see e.g.~\cite{acc:07,accm:05}, where fundamental groups are isomorphic but some special braid monodromies are not Hurwitz equivalent. The main result in~\cite{acc:03} basically implies that if two curves with the same combinatorics have non-Hurwitz equivalent
braid monodromies, then, by adding the non-transversal lines through the base point, we obtain Zariski pairs. Actually, we obtain that the 
space $\Sigma$ of the curves with this combinatorics is not connected (a consequence of Proposition~\ref{prop:isotopy}).

As it was pointed out in the Introduction, computing braid monodromies is usually a difficult task. Once they are computed, checking if they are 
Hurwitz equivalent is also hard. In \cite{acc:07,accm:05} it was achieved using finite representations of the braid group and replacing the Hurwitz
action on $(\mathbb{B}_n)^r$ by the action of $G\times\mathbb{B}_r$ on $G^r$, where $G$ is a finite group. Note that the target
of these representations
must be bigger than the corresponding symmetric groups, so, it becomes very hard when $n$ is big, e.g., greater than~$6$.

There are techniques to obtain braid monodromies of a curve if we know one braid monodromy factorization for a simple one. Let 
us illustrate it with Kummer covers, maps $\Psi_m:\bp^2\to\bp^2$ given by 
$\Psi_m([X:Y:Z])=[X^m:Y^m:Z^m]$. Given a projective curve $\bar{C}$ of degree~$d$ (the axes are not contained in the curve), $\bar{C}_m:=\Psi_m^{-1}(\bar{C})$ is a curve of degree~$n:=d\cdot m$. For the sake of simplicity, let us assume that $\bar{C}$ is transversal to the axes.
The following result will be useful.

\begin{prop}[\cite{o_s:78}]
Let $C_1,C_2\subset\bc^2$ be two affine curves of degrees~$n_1,n_2$, respectively, and such that $C_1\cap C_2$ contains exactly $n_1\cdot n_2$
points. Then
\[
\pi_1(\bc^2\setminus(C_1\cup C_2))\cong
\pi_1(\bc^2\setminus C_1)\times\pi_1(\bc^2\setminus C_2).
\]
\end{prop}

\begin{cor}
Let $\bar{C}\subset\bp^2$ be a curve transversal to the union of the three axes, and let $G$ be the fundamental group of the complement
of $\bar{C}\cup\{Z=0\}$. Then, the fundamental group of the complement
of $\bar{C}\cup\{X\cdot Y\cdot Z=0\}$ is isomorphic to $G\times\mathbb{Z}^2$. There are two meridians of $X=0$ and $Y=0$
which are the generators of $\mathbb{Z}^2$.
\end{cor}

Using covering theory, the following result is not difficult; it is based on the ideas appearing in \cite{Uludag:2001}.

\begin{prop}
With the above transversality hypotheses, the fundamental group of
the complement of $\bar{C}_m\cup\{X\cdot Y\cdot Z=0\}$ is isomorphic to a subgroup of the fundamental
group of the complement of $\bar{C}\cup\{X\cdot Y\cdot Z=0\}$. Under the above isomorphism it is $G\times (m\mathbb{Z})^2$
and there are two meridians of $X=0$ and $Y=0$
which are the generators of $(m\mathbb{Z})^2$.
\end{prop}

The following result is a consequence of \eqref{eq:pb}.

\begin{prop}
With the above transversality hypotheses, there is a central short exact sequence
\[
0\to\mathbb{Z}/m\to\pi_1(\bp^2\setminus\bar{C}_m)\to\pi_1(\bp^2\setminus\bar{C})\to 1.
\]
\end{prop}

In general one can find a braid monodromy factorization (maybe non generic) of $\bar{C}_m$ in terms of a braid monodromy factorization of $\bar{C}$,
see~\cite{aco:2014}. In the above case of transversality, it is easier. One can assume that we can take $L=\{Z=0\}$ and $P=[0:1:0]$ (it is enough to deform the axes slightly). Each line through $P$ (different from $X=0$, $Z=0$) has as preimage $m$ lines through~$P$. Each not transversal
line through $P$ for $\bar{C}$ produces $m$ non-transversal lines for $\bar{C}_m$; moreover, if the braid monodromy for $\bar{C}$ is generic,
each one of these lines contains $m$ non-transversal points. Moreover, each one of the $m$ lines passing through $P$ and 
one point of $\bar{C}\cap \{Y=0\}$ is lifted to $m$ non-transversal lines, each one is tangent to $\bar{C}_m$ at a point with intersection
multiplicity~$m$. 

This new braid monodromy is not generic, but it is not hard to deform it into a generic braid monodromy factorization which can be
recovered only in terms of the braid monodromy factorization of~$\bar{C}$. 

Since the eighties the study of the theory of complex analytic and algebraic varieties has been enriched by the study of pseudo-holomorphic and complex symplectic
varieties and one important case is that of symplectic curves in complex surfaces,
namely in $\bp^2$. Symplectic curves 
in~$\bp^2$ also admit braid monodromy factorizations coming from 
pencils of pseudo-holomorphic lines. 

This gives a technique to construct 
symplectic curves. Namely, consider an algebraic plane curve in $\bp^2$
and perform a small deformation where the singular points may be~\emph{broken}
in several points having the topological type of algebraic singularities, e.g., deform 
an ordinary triple point into three ordinary double points. This deformation induces also a deformation of a generic algebraic
pencil of lines into a pseudo-holomorphic pencil of lines from which a 
braid monodromy factorization for the symplectic curve can be obtained. In the previous example,
the braid monodromy factorization of an ordinary triple point is locally given 
by $((\sigma_1\cdot\sigma_2)^3)$; the braid monodromy factorization of the deformed
ordinary double points is $(\sigma_1^2,\sigma^{-1}\cdot\sigma_2^2\cdot\sigma_1,\sigma_2^2)$.
The braid monodromy factorization of the deformed symplectic curve can be
seen as a deformation of the braid monodromy factorization of the algebraic curve. 
We can state the following question: Given 
a symplectic curve, is it isotopic to an algebraic one?

For example, in the case of smooth curves, it is known that any symplectic smooth curve up to degree~$18$
is isotopic to an algebraic one. Probably the first concrete example of an irreducible symplectic curve which cannot be isotopic
to an algebraic one is due to Orevkov.

The details of Orevkov's example can be found in~\cite{Golla-Starkston:22}, but we sketch the construction.
Start with the tricuspidal quartic and 
the tangent lines to the cusps. It is known that these lines are concurrent. Deform them in the symplectic category such that 
they are not concurrent any more, as we did in the example of deformation.
We perform a symplectic standard Cremona transformation on these lines. The result is 
a symplectic curve of degree~$8$, with three singular points with the same
topological type as $v(u^3 - v^5)=0$. Such curves do not exist in the algebraic category
(actually using this argument backwards gives a proof), and, in particular, the constructed curve cannot be isotopic to an algebraic one.
With the same ideas we may consider the symplectic non-Pappus arrangement (obtained by deforming the line joining the three points in Pappus Theorem into a \emph{generic} passing only through two of them, see Figure~\ref{fig:pappus}).

\begin{figure}[ht]
\centering
﻿\begin{tikzpicture}[scale=.625,vertice/.style={draw,circle,fill,minimum size=0.2cm,inner sep=0}]
\coordinate (A1) at (-.5,0);
\coordinate (A2) at (2,0);
\coordinate (A3) at ($-1.5*(A1)+2.5*(A2)$);
\coordinate (B1) at (-1,1.75);
\coordinate (B2) at (2,3.5);
\coordinate (B3) at ($-1*(B1)+2*(B2)$);

\draw[line width=1.5] ($1.75*(A1)-.75*(A3)$)--($-.25*(A1)+1.25*(A3)$);
\draw[line width=1.5] ($2.3*(B1)-1.3*(B3)$)--($-.25*(B1)+1.25*(B3)$);

\node[vertice] at (A1) {};
\node[below=5pt] at (A1) {$A_1$};
\node[vertice] at (A2) {};
\node[below=5pt] at (A2) {$A_2$};
\node[vertice] at (A3) {};
\node[below=5pt] at (A3) {$A_3$};
\node[vertice] at (B1) {};
\node[above=5pt] at (B1) {$B_1$};
\node[vertice] at (B2) {};
\node[above=5pt] at (B2) {$B_2$};
\node[vertice] at (B3) {};
\node[above=5pt] at (B3) {$B_3$};

\draw[name path=L12, line width=1.2] ($1.25*(A1)-.25*(B2)$)--($-.25*(A1)+1.25*(B2)$);
\draw[name path=M12, line width=1.2] ($1.5*(B1)-.5*(A2)$)--($-.25*(B1)+1.25*(A2)$);
\path [name intersections={of=L12 and M12,by=R}];

\node[vertice] at (R) {};

\draw[name path=L13, line width=1.2] ($1.25*(A1)-.25*(B3)$)--($-.25*(A1)+1.25*(B3)$);
\draw[name path=M13, line width=1.2] ($1.5*(B1)-.5*(A3)$)--($-.25*(B1)+1.25*(A3)$);
\path [name intersections={of=L13 and M13,by=Q}];

\node[vertice] at (Q) {};

\draw[name path=L23, line width=1.2] ($1.25*(A2)-.25*(B3)$)--($-.25*(A2)+1.25*(B3)$);
\draw[name path=M23, line width=1.2] ($1.5*(B2)-.5*(A3)$)--($-.25*(B2)+1.25*(A3)$);
\path [name intersections={of=L23 and M23,by=P}];

\node[vertice] at (P) {};

\draw[line width=1.2] ($1.25*(P)-.25*(R)$) to[out=90] ($.75*(P)+.25*(R)$) -- ($4*(R)-3*(P)$);
\node[vertice] at (P) {};
\node[vertice] at (Q) {};
\node[vertice] at (R) {};

\end{tikzpicture}
 \caption{Non-Pappus symplectic arrangement}
\label{fig:pappus}
\end{figure}
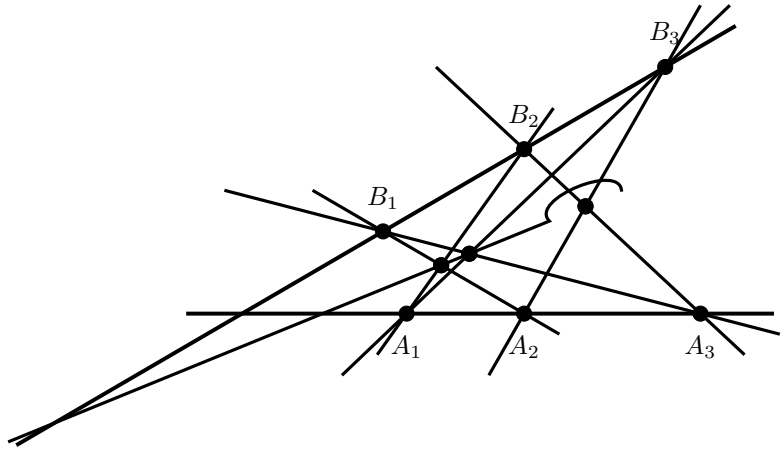

Braid monodromy factorizations are a key invariant to study symplectic curves. In the previous paragraph
Orevkov's example has been described, but there are infinitely many previous examples prior to this one. 
They are in the outstanding paper of B.~Moishezon~\cite{Moi:94}. 
Moishezon gave a general construction, starting from a braid monodromy factorization $\boldsymbol{\tau}$ of a curve~$\bar{C}$ of 
degree~$d$. 

For any $m\geq 3$, he constructs an infinite family of braid monodromy factorizations $\boldsymbol{\tau}_{\boldsymbol{\sigma}}$,
where $\boldsymbol{\sigma}$ are tuples of $\frac{(m-1)\cdot(m-2)}{2}$ braids in $\mathbb{B}_d$.

The main feature of $\boldsymbol{\tau}_{\boldsymbol{\sigma}}$ is that 
its factors have the conjugacy class of the factors of a braid monodromy factorization of $\bar{C}_m$. More precisely,
if $\boldsymbol{\sigma}$ contains only the identity, it can be proved that $G^\aff(\boldsymbol{\tau}_{\boldsymbol{\sigma}})$
is isomorphic to $\pi_1(\bp^2\setminus(\bar{C}_m\cup\{Z=0\}))$, so it is conceivable that they are isotopic (though I do not have 
a proof). Of course, it may happen that this infinite family lives in a Hurwitz orbit.

Besides this general construction, he finds some curves $\bar{C}_d$, starting from a sextic curve with~$9$ cusps, for which he finds
an infinite family of $\boldsymbol{\sigma}$'s such that the groups  $G^\aff(\boldsymbol{\tau}_{\boldsymbol{\sigma}})$
are pairwise non-isomorphic. As a consequence, their braid monodromy factorizations are not Hurwitz-equivalent.
Let $\Sigma$ be the space of algebraic curves with this combinatorics. Since it is a quasi-projective variety, it has at most
a finite number of irreducible components, and thus a finite number of connected components. At most a finite number of 
$\boldsymbol{\tau}_{\boldsymbol{\sigma}}$ can be braid monodromy factorizations of algebraic curves. Hence, he finds an infinite
number of symplectic curves which are non-isotopic to algebraic curves, even though no explicit example is shown.
 
%

\newcommand{\etalchar}[1]{$^{#1}$}
\providecommand{\noopsort}[1]{}
\providecommand{\bysame}{\leavevmode\hbox to3em{\hrulefill}\thinspace}
\providecommand{\MR}{\relax\ifhmode\unskip\space\fi MR }
\providecommand{\MRhref}[2]{%
	\href{http://www.ams.org/mathscinet-getitem?mr=#1}{#2}
}
\providecommand{\href}[2]{#2}

\end{document}